\documentclass[12pt,draftcls,onecolumn]{IEEEtran}
\usepackage{graphics} 
\usepackage{epsfig} 
\usepackage{amsmath} 
\usepackage{amsthm}
\usepackage{amssymb}  

\usepackage{mathrsfs}

\usepackage{subfigure}
\usepackage{color} 
\usepackage{cite}
\usepackage{mathtools}
\usepackage{mathrsfs}
\usepackage{txfonts}
\usepackage{amssymb}
\usepackage{multirow}
\usepackage{subfigure}

\newtheorem{assumption}{Assumption}[section]

\newcommand{\HEI}[1]{\bf Hybrid-EI}

\newtheorem{theorem}{Theorem}
\newtheorem{lemma}{Lemma}
\newtheorem{remark}{Remark}

\ifCLASSINFOpdf

\else

\fi

\begin{document}
%

\title{{A Note on Stability of Event-Triggered Control Systems with Time Delays}}

%
%
%

\author{Kexue~Zhang~~~~
        ~Bahman~Gharesifard~~~~
        ~Elena~Braverman
\thanks{This work was supported by the Natural Sciences and Engineering Research Council of Canada (NSERC) under grants RGPIN-2020-03934 and PGPIN-2022-03144.}
\thanks{K. Zhang is with the Department of Mathematics and Statistics, Queen's University, Kingston, Ontario K7L 3N6, Canada (e-mail: kexue.zhang@queensu.ca).

B. Gharesifard is with the Department of Electrical and Computer Engineering, University of California, Los Angeles, CA 90095, USA (e-mail: gharesifard@ucla.edu).

E. Braverman is with the Department
of Mathematics and Statistics, University of Calgary, Calgary, Alberta T2N 1N4, Canada (e-mail: maelena@ucalgary.ca).}
}

\maketitle
%
\begin{abstract}
This note studies stability of event-triggered control systems with the event-triggered control algorithm proposed in~\cite{KZ-BG-EB:2022}. { We construct a novel Halanay-type inequality, which is used to show that sufficient conditions of the main results in~\cite{KZ-BG-EB:2022} ensure stability of the event-triggered control systems that was missing in~\cite{KZ-BG-EB:2022}. It is also shown that a positive parameter in the proposed event-triggering condition in~\cite{KZ-BG-EB:2022} can be freely selected to exclude Zeno behavior from the event-triggered control system. An illustrative example is investigated to demonstrate the theoretical results of this study with numerical simulations. }
\end{abstract}
%

\section{Introduction}\label{Sec1}

In~\cite{KZ-BG-EB:2022}, an event-triggered control algorithm was proposed for nonlinear time-delay systems. { In the designed event-triggering condition, the exponential function $a e^{-b(t-t_0)}$ with tunable parameters $a$ and $b$} plays a vital role in ensuring boundedness and attractivity of the system states while excluding Zeno behavior, a phenomenon that the control updates are triggered infinitely many times over a finite time period. However, stability criterion with the proposed event-triggered control algorithm was not established in~\cite{KZ-BG-EB:2022} {when $a>0$}. 

In this note, we will show that event-triggered control systems with positive $a$ in~\cite{KZ-BG-EB:2022} actually are stable if the sufficient conditions in Theorem~2 (or Theorem~3) of~\cite{KZ-BG-EB:2022} are satisfied. { We will also prove that the positive parameter $b$ can be chosen arbitrarily, whereas the results in~\cite{KZ-BG-EB:2022} require the positive $b$ to be upper bounded, in order to exclude Zeno behavior from the control system.

The rest of this note is organized as follows. In Section~\ref{Sec1}, we rephrase the event-triggered control problem for general nonlinear time-delay systems considered in~\cite{KZ-BG-EB:2022}. To show stability of the event-triggered control systems, a new Halanay-type inequality is introduced in Section~\ref{Sec3}. Main results are presented in Section~\ref{Sec4} with some remarks to further elaborate the improvements achieved in this study. We illustrate the main results by a numerical example in Section~\ref{Sec5}, and finally draw conclusions in Section~\ref{Sec6}. }

{
\section{Problem Formulation}\label{Sec2}
To make this note self-contained, we adopt the notations from~\cite{KZ-BG-EB:2022} and briefly introduce the event-triggered control problem formulated in~\cite{KZ-BG-EB:2022} in this section.

Denote by $\mathbb{N}$ the set of positive integers, $\mathbb{R}$ the set of real numbers, $\mathbb{R}^+$ the set of nonnegative reals, and $\mathbb{R}^n$ the $n$-dimensional real space equipped with the Euclidean norm denoted by $\|\cdot\|$. For $a,b\in \mathbb{R}$ with $b>a$, we define
\begin{align*}
\mathcal{PC}([a,b],\mathbb{R}^n)=& \{ \phi:[a,b]\rightarrow\mathbb{R}^n \mid \phi \textrm{ is piecewise} \textrm{  right-}\cr
& \textrm{continuous} \}\cr
\mathcal{PC}([a,\infty),\mathbb{R}^n) =&  \{\phi:[a,\infty)\rightarrow\mathbb{R}^n \mid \phi|_{[a,c]}\in \mathcal{PC}([a,c],\mathbb{R}^n) \cr
 & \textrm{ for all } c>a \}
\end{align*}
where $\phi|_{[a,c]}$ is a restriction of $\phi$ on interval $[a,c]$. Let $\mathcal{C}(J,\mathbb{R}^n)$ denote the set of continuous functions from interval $J$ to $\mathbb{R}^n$. Given $\tau>0$, the linear space $\mathcal{C}([-\tau,0],\mathbb{R}^n)$ is equipped with the supremum norm $\|\phi\|_{\tau}:=\sup_{s\in[-\tau,0]}\|\phi(s)\|$ for $\phi\in \mathcal{C}([-\tau,0],\mathbb{R}^n)$. A function $\alpha:\mathbb{R}^+\rightarrow\mathbb{R}$ is said to be of class $\mathcal{K}$ and we write $\alpha\in\mathcal{K}$, if $\alpha$ is continuous, strictly increasing, and satisfies $\alpha(0)=0$. If $\alpha\in\mathcal{K}$ and also $\alpha(s)\rightarrow\infty$ as $s\rightarrow\infty$, we say that $\alpha$ is of class $\mathcal{K}_{\infty}$ and we write $\alpha\in\mathcal{K}_{\infty}$. A continuous function $\beta:\mathbb{R}^+\times\mathbb{R}^+ \rightarrow\mathbb{R}^+$ is said to be of class $\mathcal{KL}$ and we write $\beta\in \mathcal{KL}$, if the function $\beta(\cdot,t)$ is of class $\mathcal{K}$ for each fixed $t\in\mathbb{R}^+$, and the function $\beta(s,\cdot)$ is decreasing and $\beta(s,t)\rightarrow 0$ as $t\rightarrow \infty$ for each fixed $s\in \mathbb{R}^+$.

Next, we recall the event-triggered control problem in~\cite{KZ-BG-EB:2022}. Consider the following sampled-data control system
\begin{eqnarray}\label{ETC.system}
\left\{\begin{array}{ll}
\dot{x}(t)=f(t,x_t,u(t)) \cr
u(t)=k(x(t_i)),{~t\in[t_{i},t_{i+1})}\cr
x_{t_0}=\varphi
\end{array}\right.
\end{eqnarray}
where $x\in\mathbb{R}^n$ is the system state; control input $u\in\mathbb{R}^m$ is the state feedback control regulated by the feedback control law $k:\mathbb{R}^n\mapsto \mathbb{R}^m$ satisfying $k(0)=0$; the sampling time sequence $\{t_i\}_{i\in\mathbb{N}}$ is to be determined by a certain triggering condition from~\cite{KZ-BG-EB:2022} which will be introduced later; $\varphi\in\mathcal{C}([-\tau,0],\mathbb{R}^n)$ represents the initial function; $f:\mathbb{R}^+\times\mathcal{C}([-\tau,0],\mathbb{R}^n)\times\mathbb{R}^m\mapsto \mathbb{R}^n$ satisfies $f(t,0,0)=0$ for all $t\in\mathbb{R}^+$, and hence system~\eqref{ETC.system} admits the zero solution; given a time $t$, the function $x_{t}$ is defined as $x_{t}(s):=x(t+s)$ for $s\in[-\tau,0]$, and $\tau>0$ represents the maximum involved delay in the system. 

Let us denote the sampling error by
\begin{equation}\label{error}
\epsilon(t)=x(t_i)-x(t), \textrm{~for~} t\in[t_i,t_{i+1}).
\end{equation}
Then system~\eqref{ETC.system} can be written as the following closed-loop system
\begin{eqnarray}\label{Closedloop}
\left\{\begin{array}{ll}
\dot{x}(t)=f(t,x_t,k(x+\epsilon)) \cr
x_{t_0}=\varphi
\end{array}\right.
\end{eqnarray}

To introduce the triggering condition for determining the time sequence~$\{t_i\}_{i\in\mathbb{N}}$, we recall some concepts related to the Lyapunov functional candidate for time-delay systems. A function $V:\mathbb{R}^+\times\mathbb{R}^n\mapsto \mathbb{R}^+$ is said to be of class $\mathcal{V}_0$ and we write $V\in\mathcal{V}_0$, if, for each $x\in\mathcal{C}(\mathbb{R}^+,\mathbb{R}^n)$, the composite function $t\mapsto V(t,x(t))$ is continuous. A functional $V:\mathbb{R}^+\times \mathcal{C}([-\tau,0],\mathbb{R}^n)\mapsto\mathbb{R}^+$ is said to be of class $\mathcal{V}^*_0$ and we write $V\in \mathcal{V}^*_0$, if, for each function $x\in\mathcal{C}([-\tau,\infty),\mathbb{R}^n)$, {the composite function $t\mapsto V(t,x_t)$ is continuous in $t$ for all $t\geq 0$}, and $V$ is locally Lipschitz in its second argument. Given an input $u\in \mathcal{PC}([t_0,\infty),\mathbb{R}^m)$, we define the upper right-hand Dini derivative of the Lyapunov functional candidate $V(t,x_t)$ with respect to system \eqref{ETC.system}:
\[
\mathrm{D}^+V(t,\phi)=\limsup_{\varepsilon\rightarrow 0^+}\frac{V\left(t+\varepsilon,x_{t+\varepsilon}(t,\phi)\right)-V(t,\phi)}{\varepsilon}
\]
where $x(t,\phi)$ denotes the solution to \eqref{ETC.system} satisfying $x_t=\phi$.

Throughout this study, we assume the closed-loop system~\eqref{Closedloop} satisfies the following conditions.
\begin{assumption}\label{Assumption}
There exist functions $V_1\in \mathcal{V}_0$, $V_2\in \mathcal{V}^*_0$, $\alpha_1, \alpha_2, \alpha_3 \in \mathcal{K}_{\infty}$ and $\chi\in \mathcal{K}$, and constant $\mu>0$ such that 
\begin{itemize}
\item[(i)] $\alpha_1(\|\phi(0)\|)\leq V_1(t,\phi(0))\leq \alpha_2(\|\phi(0)\|)$ and $0\leq V_2(t,\phi)\leq \alpha_3(\|\phi\|_{\tau})$;

\item[(ii)] $V(t,\phi):=V_1(t,\phi(0))+V_2(t,\phi)$ satisfies
\[
\mathrm{D}^+V(t,\phi) \leq -\mu V(t,\phi) +\chi(\|\epsilon\|).
\]
\end{itemize}
\end{assumption}

Now we are in the position to introduce the triggering condition in~\cite{KZ-BG-EB:2022} for system~\eqref{ETC.system}. To enforce $\epsilon$ to satisfy the condition
\begin{equation}\label{enforce}
\chi(\|\epsilon\|)\leq \sigma \alpha_1(\|x\|) + \chi\left(a e^{-b(t-t_0)}\right)
\end{equation}
where $\sigma\geq 0$, $a\geq 0$, and $b>0$ are constants, we update the control input $u$ when the following triggering condition is satisfied
\begin{equation}\label{trigger}
\chi(\|\epsilon\|) = \sigma \alpha_1(\|x\|) + \chi\left(a e^{-b(t-t_0)}\right).
\end{equation}
When condition~\eqref{trigger} holds, we say an event occurs and then a control update is executed. Therefore, the sampling times $\{t_i\}_{i\in\mathbb{N}}$ are determined as follows:
\begin{equation}\label{event.time}
t_{i+1}=\inf\left\{t\geq t_i \mid \chi(\|\epsilon\|) = \sigma \alpha_1(\|x\|)+ \chi\left(a e^{-b(t-t_0)}\right)\right\},
\end{equation}
which are also called event times. Since the event times are implicitly determined by the triggering condition~\eqref{trigger}, it is essential to exclude Zeno behavior, a phenomenon that infinitely many events happen over a finite time interval, from the closed-loop system.

It has been shown in~\cite{XL-KZ:2020} that Assumption~\ref{Assumption} implies the closed-loop system~\eqref{Closedloop} is input-to-state stable with respect to the sampling error $\epsilon$, and the state feedback control $u(t)=k(x(t))$ for $t\geq t_0$ renders the following closed-loop system
\begin{eqnarray}\label{Closedloop0}
\left\{\begin{array}{ll}
\dot{x}(t)=f(t,x_t,k(x(t))) \cr
x_{t_0}=\varphi
\end{array}\right.
\end{eqnarray}
globally asymptotically stable. The objective of this study is to show that when $a>0$, the sufficient conditions provided in~\cite{KZ-BG-EB:2022} ensures global asymptotic stability of system~\eqref{Closedloop} with the event times determined by~\eqref{event.time}, rather than just uniform boundedness and global attractivity as proved in~\cite{KZ-BG-EB:2022}, and also to show that the closed-loop system~\eqref{Closedloop} does not exhibit Zeno behavior for any positive $b$, while the results in~\cite{KZ-BG-EB:2022} require $b<\mu-\sigma$ to rule out Zeno behavior.

To show stability of the closed-loop system, we will introduce a novel Halanay-type inequality in the next section.
}

\section{Halanay-Type Inequality}\label{Sec3}

For a continuous function $g:\mathbb{R}\rightarrow\mathbb{R}$, the Dini-derivatives $\mathrm{D}^+g(t)$ and $\mathrm{D}_-g(t)$ are defined as follows:
\[
\mathrm{D}^+g(t)=\limsup_{\varepsilon\rightarrow 0^+} \frac{g(t+\varepsilon)-g(t)}{\varepsilon}
\]
and
\[
\mathrm{D}_-g(t)=\liminf_{\varepsilon\rightarrow 0^-} \frac{g(t+\varepsilon)-g(t)}{\varepsilon}.
\]

The following lemma from~\cite{CTHB-EB:2005} provides a relationship between $\textrm{D}^+g(t)$ and $\textrm{D}_-g(t)$, which will be used to construct our Halanay-type inequality.

\begin{lemma}\label{lemma}
Let $p$ and $q$ be continuous functions with $\mathrm{D}^+p(t) \leq q(t)$ for $t$ in some open interval $\mathcal{I}$. Then $\mathrm{D}_-p(t)\leq q(t)$ for $t\in\mathcal{I}$.
\end{lemma}

Next, we introduce a new Halanay-type inequality which allows us to bound the states of  the event-triggered control system with time delays.

\begin{lemma}\label{lemma.Halanay}
Let $g:[t_0-r,t_0+\Gamma)\rightarrow \mathbb{R}^+$ be a continuous function satisfying
\begin{equation}\label{Halanay}
\mathrm{D}^+g(t)\leq \gamma_1 g(t_0) +\gamma_2 \|g_t\|_r \textrm{~~~for~~~}t_0\leq t< t_0+\Gamma
\end{equation}
where $r$, $\Gamma$, $\gamma_1$, and $\gamma_2$ are positive constants. Then
\[
g(t)\leq \|g_{t_0}\|_r e^{\lambda(t-t_0)} \textrm{~~~for~~~}t_0\leq t< t_0+\Gamma
\]
 where $\lambda=\gamma_1+\gamma_2$.
\end{lemma}

\begin{proof}
Define
\begin{eqnarray*}
w(t)=\left\{\begin{array}{ll}
\|g_{t_0}\|_r e^{\lambda(t-t_0)}, &\textrm{~~if~~} t_0<t<t_0+\Gamma \cr
\|g_{t_0}\|_r,  &\textrm{~~if~~} t_0-r\leq t\leq t_0 
\end{array}\right.
\end{eqnarray*}
and let $K>1$ be an arbitrary constant. Then, for $t\in [t_0-r,t_0]$, we have
\begin{equation}\label{lessthant0}
g(t)\leq \|g_{t_0}\|_r = w(t) <K w(t),
\end{equation} 
that is, $g(t)<Kw(t)$ for  $t\in [t_0-r,t_0]$.

Next, we use a contradiction argument to show that $g(t)<Kw(t)$ for  $t\in (t_0,t_0+\Gamma)$. Suppose there exits some $t\in (t_0,t_0+\Gamma)$ such that $g(t)\geq Kw(t)$, then we define
\[
\bar{t}=\inf\left\{t\in (t_0,t_0+\Gamma) \mid g(t)\geq Kw(t)\right\}.
\]
From the continuities of $g$ and $w$, we have 
\begin{equation}\label{property1}
g(t)<Kw(t) \textrm{~~for~~} t_0<t<\bar{t}
\end{equation}
and 
\begin{equation}\label{property2}
g(\bar{t})=Kw(\bar{t}).
\end{equation}

By~\eqref{property1} and~\eqref{property2}, we conclude that
\[
\frac{g(\bar{t}+\varepsilon)-g(\bar{t})}{\varepsilon}> \frac{Kw(\bar{t}+\varepsilon)-Kw(\bar{t})}{\varepsilon}
\]
for $\varepsilon<0$ close to $0$. Hence, 
\begin{equation}\label{contradiction}
\mathrm{D}_-g(\bar{t})\geq K\dot{w}(\bar{t}).
\end{equation}

On the other hand, by Lemma~\ref{lemma} and~\eqref{Halanay}, we get
\begin{align*}
\mathrm{D}_-g(\bar{t}) &\leq \gamma_1 g(t_0) +\gamma_2 \|g_{\bar{t}}\|_r \cr
              &<  \gamma_1 Kw(t_0) +\gamma_2 K\|w_{\bar{t}}\|_r \cr
              &< (\gamma_1 + \gamma_2)K w(\bar{t}) \cr
              &= \lambda Kw(\bar{t}) \cr
              &=K\dot{w}(\bar{t})
\end{align*}
where we used~\eqref{lessthant0},~\eqref{property1},~\eqref{property2}, and the definition of $w$ in the last two inequalities. This is a contradiction to~\eqref{contradiction}. Therefore, we conclude that $g(t)<Kw(t)$ for  $t\in (t_0,t_0+\Gamma)$.

Since $K>1$ is arbitrary, we let $K\rightarrow 1$ and then $g(t)\leq w(t)$ for  $t\in [t_0,t_0+\Gamma)$, that is, the proof is completed.

\end{proof}

Discussions on this lemma will be provided in Remark~\ref{remark2} with the main results in the following section.

\section{Main Results}\label{Sec4}

{
Now we are ready to introduce the main results of this study.

\begin{theorem}\label{Theorem}
Suppose that Assumption \ref{Assumption} holds with $V_1\in \mathcal{V}_0$, $V_2\in \mathcal{V}^*_0$, $\alpha_1, \alpha_2, \alpha_3 \in \mathcal{K}_{\infty}$ and $\chi\in \mathcal{K}$, and constant $\mu>0$. The event times $\{t_i\}_{i\in\mathbb{N}}$ are defined by~\eqref{event.time} with $0\leq \sigma<\mu$, $a> 0$, and $b>0$. We further assume that
\begin{itemize}
\item[(iii)] $\alpha^{-1}_1$, $\chi$, and $k$ are locally Lipschitz, where $\alpha^{-1}_1$ denotes the inverse of $\alpha_1$;

\item[(iv)] $f$ is locally Lipschitz in its second and third arguments, respectively.

\end{itemize}
Then the closed-loop system~\eqref{Closedloop} is globally asymptotically stable and does not exhibit Zeno behavior.
\end{theorem}
}
\begin{proof} { Let
\begin{align}\label{eta}
\eta:=\left\{\begin{array}{ll}
\min\{b,\mu-\sigma\},~&\textrm{if}~ b\not=\mu-\sigma \cr
\xi,~&\textrm{if}~ b=\mu-\sigma
\end{array}\right.
\end{align}
where $\xi<b$ is a positive constant, and
\[
\bar{M}:=\left\{\begin{array}{ll}
\frac{aL}{|\mu-\sigma-b|},~&\textrm{if}~ b\not=\mu-\sigma \cr
\frac{aL}{|\mu-\sigma-\xi|},~&\textrm{if}~ b=\mu-\sigma
\end{array}\right.
\]
then from the proof of Theorem~2 in~\cite{KZ-BG-EB:2022}, we have
\begin{equation}\label{xbdd}
\|x(t)\|\leq \alpha^{-1}_1\big(M e^{-\eta(t-t_0)}\big)\leq \alpha^{-1}_1\big(M \big)
\end{equation}
for all $t\geq t_0$, where $M=\alpha_2(\|\varphi(0)\|)+\alpha_3(\|\varphi\|_{\tau})+\bar{M}.$ Global attractivity of the zero solution follows from~\eqref{xbdd}, that is, $\lim_{t\rightarrow \infty} \|x(t)\|=0$ for any initial condition $\varphi\in\mathcal{C}([-\tau,0],\mathbb{R}^n)$.

Next, we show stability of the closed-loop system~\eqref{Closedloop}. } From the system dynamics of~\eqref{ETC.system} on $[t_i,t_{i+1})$ and the Lipschitz conditions on $f$ and $k$, we have
\begin{align}\label{Halanayx}
\mathrm{D}^+\|x(t)\|\leq \|\dot{x}(t)\|&=\|f(t,x_t,k(x(t_i)))\|\cr
           &\leq L_2 \|x_t\|_{\tau} +L_3 \|x(t_i)\| 
\end{align}
which implies~\eqref{Halanay} holds on $[t_i,t_{i+1})$ with $r=\tau$, $g(t)=\|x(t)\|$, $\gamma_1=L_3$, and $\gamma_2=L_2$, where $L_2$ is the Lipschitz constant of the function $f(t,\cdot,u):\mathcal{C}([-\tau,0],\mathbb{R}^n)\mapsto \mathbb{R}^n$ on the compact set $\{\phi\in \mathcal{C}([-\tau,0] \mid \|\phi\|_{\tau}\leq R\}$, and $L_3$ is the Lipschitz constant of the composite function $f(t,\phi,k(\cdot)):\mathbb{R}^n \mapsto \mathbb{R}^n$ on the compact set $\{x\in \mathbb{R}^n \mid \|x\|\leq R\}$ with $R=\alpha^{-1}_1(M)$. 

We then conclude from Lemma~\ref{lemma.Halanay} that
\begin{equation}\label{boundx}
\|x(t)\| \leq \|x_{t_i}\|_{\tau} e^{\lambda (t-t_i)} \textrm{~~~for~~~}t_i\leq t< t_{i+1} \textrm{~and~} i\in \mathbb{Z}^+
\end{equation}
where $\lambda=L_2+L_3$ and $\mathbb{Z}^+$ denotes the set of non-negative integers. 

We next use mathematical induction to show that $\|x(t)\| \leq \|\varphi\|_{\tau} e^{\lambda (t-t_0)}$ for all $t\geq t_0$. We can see from~\eqref{boundx} that this statement holds for $t\in[t_0,t_1)$. Suppose $\|x(t)\| \leq \|\varphi\|_{\tau} e^{\lambda (t-t_0)}$ holds for $t\in [t_0,t_i)$, and we will show this inequality holds for $t\in[t_i,t_{i+1})$. By~\eqref{boundx}, we get
\begin{align*}
\|x(t)\| &\leq \|x_{t_i}\|_{\tau} e^{\lambda (t-t_i)} \cr
         &= e^{\lambda (t-t_i)} \sup_{s\in[-\tau,0]}\|x(t_i+s)\|\cr
         &\leq e^{\lambda (t-t_i)}  \|\varphi\|_{\tau} e^{\lambda (t_i-t_0)}\cr
         &= \|\varphi\|_{\tau} e^{\lambda (t-t_0)}  \textrm{~~~for~~~}t_i\leq t< t_{i+1},
\end{align*}
that is, $\|x(t)\| \leq \|\varphi\|_{\tau} e^{\lambda (t-t_0)}$ for all $t\in [t_0,t_{i+1})$. By induction, we conclude that
\begin{equation}\label{bound}
\|x(t)\| \leq \|\varphi\|_{\tau} e^{\lambda (t-t_0)} \textrm{~~~for~all~~}t\geq t_0.
\end{equation}
Then~\eqref{bound} and~\eqref{xbdd} imply
\begin{equation}\label{dbound1}
\alpha_1(\|x(t)\|)\leq \min\left\{ \alpha_1\left(\|\varphi\|_{\tau} e^{\lambda (t-t_0)}\right), M e^{-\eta(t-t_0)} \right\} \textrm{~~~for~all~~}t\geq t_0.
\end{equation}
Let $\delta_1=\inf\{s\geq 0: \alpha_2(s)+\alpha_3(s)\geq \bar{M}\}$, then $\|\varphi\|_{\tau}<\delta_1$ implies $M=\alpha_2(\|\varphi(0)\|)+\alpha_3(\|\varphi\|_{\tau})+\bar{M}<2\bar{M}$ and 
\begin{equation}\label{dbound2}
\alpha_1(\|x(t)\|)\leq \min\left\{ \alpha_1\left(\|\varphi\|_{\tau} e^{\lambda (t-t_0)}\right), 2\bar{M} e^{-\eta(t-t_0)} \right\} \textrm{~~~for~all~~}t\geq t_0.
\end{equation}
Since $\alpha_1$ is strictly increasing, we have that $\|\varphi\|_{\tau}<\delta_2$ implies $\alpha_1(\|\varphi\|_{\tau})< 2\bar{M}$ where $\delta_2=\alpha_1^{-1}(2\bar{M})$.

Next, we consider the initial function $\varphi$ satisfying $\|\varphi\|_{\tau}<\min\{\delta_1,\delta_2\}$. Then, there exists a unique $\hat{t}>t_0$ such that 
\[
\alpha_1\left(\|\varphi\|_{\tau}e^{\lambda(\hat{t}-t_0)}\right)= 2\bar{M} e^{-\eta(\hat{t}-t_0)},
\]
where we used the fact that both $\alpha_1\left(\|\varphi\|_{\tau}e^{\lambda(t-t_0)}\right)$ and $2\bar{M}e^{-\eta(t-t_0)}$ are strictly monotonic in $t$. Furthermore, for any $\varepsilon>0$, there exists a $\delta_3$, depending on $\varepsilon$, such that $\|\varphi\|_\tau< \delta_3$ implies
\[
\alpha_1\left(\|\varphi\|_{\tau}e^{\lambda(\hat{t}-t_0)}\right)= 2\bar{M} e^{-\eta(\hat{t}-t_0)}<\alpha_1(\varepsilon),
\]
that is, small enough $\|\varphi\|_\tau$ leads to large enough $\hat{t}$ so that $2\bar{M} e^{-\eta(\hat{t}-t_0)}< \alpha_1(\varepsilon)$. Note that $\bar{M}$ is independent of $\|\varphi\|_{\tau}$.

Now we can conclude from~\eqref{dbound2} that, for any $\varepsilon>0$, there exists a $\delta=\min\{\delta_1,\delta_2,\delta_3\}$ such that $\|\varphi\|_\tau< \delta$ implies
\[
\alpha_1(\|x(t)\|)\leq  2\bar{M} e^{-\eta(\hat{t}-t_0)}< \alpha_1(\varepsilon)
\]
for $t\geq t_0$, that is, $\|x(t)\|< \varepsilon$. The stability proof is completed.

{
Therefore, if $a>0$ in the triggering condition~\eqref{trigger}, we can conclude from stability and global attractivity of the zero solution that the closed-loop system~\eqref{Closedloop} is globally asymptotically stable.

Last but not least, we show that system~\eqref{Closedloop} does not exhibit Zeno behavior for any $b>0$. It has been shown in~\cite{KZ-BG-EB:2022} that the inter-event times $\{t_i-t_{i-1}\}_{i\in\mathbb{N}}$ are lower bounded by a positive quantity when $b<\mu-\sigma$, that is, system~\eqref{Closedloop} does not exhibit Zeno behavior. Hence, we will focus on the scenario of $b\geq \mu-\sigma$.

It follows from the proof of Theorem~2 in~\cite{KZ-BG-EB:2022} that
\begin{equation}\label{ineq1}
a e^{-b(t_{i+1}-t_0)} \leq \lambda_1 \left(e^{-\eta(t_i-t_0)} - e^{-\eta(t_{i+1}-t_0)} \right)  + \lambda_2 \left(t_{i+1}-t_i \right)e^{-\eta(t_i-t_0)}
\end{equation}
where $\lambda_1={L_1L_2Me^{\eta\tau}}/{\eta}>0$, $\lambda_2=L_1L_3 M>0$, and $L_1$ is the Lipschitz constant of $\alpha^{-1}_1$ on the interval $[0,M]$. Let $T_{i+1}=t_{i+1}-t_i$, then multiplying both sides of~\eqref{ineq1} by $e^{\eta(t_i-t_0)}$ yields
\begin{equation}\label{ineq2}
a e^{-bT_{i+1}} e^{(\eta-b)(t_{i}-t_0)} \leq \lambda_1 \left(1 - e^{-\eta T_{i+1}} \right)  + \lambda_2 T_{i+1}.
\end{equation}
Next we use contradiction argument to show that system~\eqref{Closedloop} is free of Zeno behavior. Suppose that there exists a $\bar{t}<\infty$ such that $\lim_{i\rightarrow \infty} t_i=\bar{t}$, that is, $t_i<\bar{t}$ for all $i\in\mathbb{N}$. The fact $b\geq \mu-\sigma$ and the definition of $\eta$ imply $\eta<b$. It then follows from~\eqref{ineq2} that 
\begin{equation}\label{ineq3}
a e^{-bT_{i+1}} e^{(\eta-b)(\bar{t}-t_0)} \leq \lambda_1 \left(1 - e^{-\eta T_{i+1}} \right)  + \lambda_2 T_{i+1}.
\end{equation}
Define a function
\[
g(T)=a e^{-bT} e^{(\eta-b)(\bar{t}-t_0)} - \lambda_1 \left(1 - e^{-\eta T} \right) - \lambda_2 T
\]
for $T\geq 0$. It can be observed that $g(0)=a e^{(\eta-b)(\bar{t}-t_0)}>0$, $\lim_{T\rightarrow \infty} g(T)=-\infty$, and
\[
g'(T)=-b a e^{-bT} e^{(\eta-b)(\bar{t}-t_0)} - \eta \lambda_1  e^{-\eta T}  - \lambda_2 <0.
\]
Thus, $g(T)$ is a strictly decreasing function, and then the equation $g(T)=0$ has a unique solution $T^*>0$. Moreover, $g(T)<0$ if $T>T^*$. On the other hand, it follows from~\eqref{ineq3} that $g(T_{i+1})<0$, and then $T_{i+1}=t_{i+1}-t_i >T^*$, that is, the inter-event times $\{t_{i+1}-t_i\}_{i\in\mathbb{N}}$ are bounded by $T^*>0$ from below.
This contradicts to the assumption $\lim_{i\rightarrow \infty} t_i=\bar{t}<\infty$. Therefore, system~\eqref{Closedloop} with~\eqref{event.time} does not exhibit Zeno behavior when $b\geq \mu-\sigma$, which completes the proof.
}
\end{proof}

\begin{remark}\label{remark1} 
It should be mentioned that the above proof does not reply on the Lipschitz condition on $\alpha_1^{-1}$. Nevertheless, if $\alpha^{-1}_1$ is locally Lipschitz, then $\hat{t}$ and $\delta$ can be obtained explicitly. To be more specific, suppose $\alpha^{-1}_1$ is locally Lipschitz, and we can show stability as follows. Combining~\eqref{bound} and~\eqref{xbdd} yields
\begin{equation}\label{dbound}
\|x(t)\|\leq \min\left\{ \|\varphi\|_{\tau} e^{\lambda (t-t_0)}, L_1M e^{-\eta(t-t_0)} \right\} \textrm{~~~for~all~~}t\geq t_0,
\end{equation}
where $L_1$ is the Lipschitz constant of $\alpha_1^{-1}$ on interval $[0,M]$. Consider $\|\varphi\|_{\tau}<\min\{\delta_1,\bar{\delta}_2\}$ with $\bar{\delta}_2=\alpha_1^{-1}(2L_1\bar{M})$, then $M<2\bar{M}$, $\|\varphi\|_{\tau}<2L_1\bar{M}$, and there exists a unique $\hat{t}>t_0$ such that 
\[
\|\varphi\|_{\tau} e^{\lambda (\hat{t}-t_0)}= 2L_1\bar{M} e^{-\eta(\hat{t}-t_0)},
\]
and then $\hat{t}$ can be derived as
\[
\hat{t}=\frac{\ln\left(\frac{2L_1\bar{M}}{\|\varphi\|_{\tau}} \right)}{\lambda+\eta} +t_0.
\]
By~\eqref{dbound} and the definition of $\hat{t}$, we have
\begin{align}\label{stability}
\|x(t)\|&<   2L_1\bar{M} e^{-\eta(\hat{t}-t_0)} \cr
        &=    2L_1\bar{M} \exp\left(\frac{-\eta}{\lambda+\eta} \ln\left( \frac{2L_1\bar{M}}{\|\varphi\|_{\tau}} \right)  \right) \cr
        & =   \|\varphi\|_{\tau}^{\frac{\eta}{\lambda+\eta}} \left(2L_1\bar{M}\right)^{\frac{\lambda}{\lambda+\eta}}
\end{align}
for all $t\geq t_0$. For any $\varepsilon>0$, let $\delta=\min\{\delta_1,\bar{\delta}_2,\delta_3\}$ with $\delta_3=\varepsilon^{(\lambda+\eta)/\eta} (2L_1\bar{M})^{-\lambda/\eta}$. For $\|\varphi\|_{\tau}<\delta$, we can derive from~\eqref{stability} that
\[
\|x(t)\|< \|\varphi\|_{\tau}^{\frac{\eta}{\lambda+\eta}} \left(2L_1\bar{M}\right)^{\frac{\lambda}{\lambda+\eta}}< \varepsilon
\]
for $t\geq t_0$, that is, the closed-loop system~\eqref{Closedloop} is stable. It can be seen that $\delta$ is given explicitly since $\delta_1$, $\bar{\delta}_2$, and $\delta_3$ are specifically defined. 
\end{remark}

\begin{remark}\label{remark2}
The upper bound $\alpha_1^{-1}(M e^{-\eta(t-t_0)})$ of the state norm in~\eqref{xbdd} guarantees attractivity of the closed-loop system. However, $M=\alpha_2(\|\varphi(0)\|)+\alpha_3(\|\varphi\|_{\tau}) + \bar{M}$ depends not only on the initial function $\varphi$ but also on parameter $a$ in $\bar{M}$. Since $\bar{M}$ is independent of the initial function $\varphi$, the stability criterion of the event-triggered control system couldn't be derived from this upper bound solely. The role of Lemma~\ref{lemma.Halanay} is to provide another bound in~\eqref{bound} for $\|x\|$. Combining these two bounds in~\eqref{dbound1} or~\eqref{dbound} allows stability analysis for the closed-loop system. Lemma~\ref{lemma.Halanay} is different from the existing Halanay-type inequalities (see, e.g.,~\cite{CTHB-EB:2005,PP:2022}) in the following sense. In the existing Halanay-type inequalities, the Dini derivative $\mathrm{D}^+g(t)$ is bounded by the sum of a function of $g(t)$ and a function of $\|g_t\|_r$, while in Lemma~\ref{lemma.Halanay} we bound $\mathrm{D}^+g(t)$ by a linear combination of $g(t_0)$ and $\|g_t\|_r$. This major difference in the dependence of $g$ at the initial time $t_0$ allows the estimation of the state bound over each interval $[t_i,t_{i+1})$ since the control input is unchanged during two consecutive event times. 
\end{remark}

{
For quadratic forms of $V_1$, the inverse of the $\mathcal{K}$ class function $\alpha_1$ in Assumption~\ref{Assumption}(i) is not locally Lipschitz in its domain. Hence, Theorem~\ref{Theorem} cannot be applied for such types of Lyapunov candidates. Nevertheless, the following result allows us to use quadratic forms of Lyapunov function as $V_1$ in the Lyapunov candidate $V$.
\begin{theorem}\label{Theorem2}
Suppose all the conditions of Theorem~\ref{Theorem} are satisfied, and the Lipschitz assumption on $\alpha^{-1}_1$ is replaced by the following condition:
\begin{itemize}
\item $\alpha^{-1}_1$ is Lipschitz on any closed and bounded sub-interval of $(0,\infty)$.
\end{itemize}
Then the closed-loop system~\eqref{Closedloop} with the event times determined by~\eqref{event.time} is globally asymptotically stable and does not exhibit Zeno behavior.
\end{theorem}
\begin{proof}
From the discussion in Remark~\ref{remark1}, we can conclude that the proof of Theorem~\ref{Theorem} also implies global asymptotic stability of the closed-loop system under the conditions of Theorem~\ref{Theorem2}. When $b<\mu-\sigma$, the non-existence of Zeno behavior has been shown in~\cite{KZ-BG-EB:2022}. If $b\geq \mu-\sigma$, the exclusion of Zeno behavior is identical to the contradiction argument in the proof of Theorem~\ref{Theorem}. Therefore, the detailed proof is omitted.
\end{proof}

\begin{remark}\label{remakr3}
Compared with the results in~\cite{KZ-BG-EB:2022}, the improvements that Theorems~\ref{Theorem} and~\ref{Theorem2} have achieved for $a>0$ are as follows. Under the sufficient conditions established in~\cite{KZ-BG-EB:2022}, the closed-loop system~\eqref{Closedloop} with the sequence of event times determined by~\eqref{event.time} is guaranteed to be globally asymptotically stable, while the results in~\cite{KZ-BG-EB:2022} only showed the closed-loop system is uniformly bounded and globally attractive. Furthermore, the positive parameter $b$ can be chosen arbitrarily to rule out Zeno behavior from the closed-loop system, whereas the results in~\cite{KZ-BG-EB:2022} need $b<\mu-\sigma$ to exclude Zeno behavior. In summary, if Assumption~\ref{Assumption} and conditions (iii), (iv) of Theorem~\ref{Theorem} hold, then the tunable parameter $\sigma$ can be chosen so that $\sigma<\mu$, and both parameters $a$ and $b$ can be selected freely in the triggering condition~\eqref{trigger}. We refer the reader to~\cite{KZ-BG-EB:2022} for a detailed discussion about the effects of different parameter selections on the dynamic performance of the event-triggered control system.

\end{remark}
}

{
\section{An Illustrative Example}\label{Sec5}

Consider the following nonlinear time-delay control system
\begin{eqnarray}\label{example}
\left\{\begin{array}{ll}
\dot{x}_1(t)=x_2(t)+\omega_0 x^2_1(t) x_2(t)+u(t) \cr
\dot{x}_2(t)=-\omega_2 x_1(t) - \omega_1 x_2(t)- \omega_3 x_1(t-1) -\omega_2 x^3_1(t)
\end{array}\right.
\end{eqnarray}
where $x(t)=(x_1(t),x_2(t))^\top\in\mathbb{R}^2$, $\omega_i$ with $i=0,1,2,3$ are non-negative constants, and $u(t)=-x_1(t)$ is the feedback control. System~\eqref{example} has been widely used to model the machine tool chatter in the cutting process (see, e.g., \cite{YA-GS-EB-TS-ZMK:2020} and references therein). In this example, we consider the following parameters: $\omega_0=1$, $\omega_1=0.5$, $\omega_2=1$, and $\omega_3=0.3$.

To show system~\eqref{example} with the given feedback control is asymptotically stable and to design the event-triggered control implementation, we consider the Lyapunov functional $V(t)=V_1(t)+V_2(t)$ with
\[
V_1(t)=x^\top(t) x(t)
\]
and
\[
V_2(t)=\delta \int^t_{t-1} e^{-\zeta(t-s)} x^\top(s)x(s)\textrm{d}s
\]
where $\delta=0.4$ and $\zeta=0.28$. It can be seen that Assumption~\ref{Assumption}(i) holds with $\alpha_1(s)=s^2$, and its inverse function is not locally Lipschitz but Lipschitz on any closed and bounded sub-interval of $(0,\infty)$.

Under the sampled-data implementation, system~\eqref{example} can be written as the following closed-loop system
\begin{eqnarray}\label{example1}
\left\{\begin{array}{ll}
\dot{x}_1(t)=x_2(t)+\omega_0 x^2_1(t) x_2(t)-x_1(t) - \epsilon_1(t) \cr
\dot{x}_2(t)=-\omega_2 x_1(t) - \omega_1 x_2(t)- \omega_3 x_1(t-1) -\omega_2 x^3_1(t)
\end{array}\right.
\end{eqnarray}
where $\epsilon_1(t)=x_1(t_i)-x_1(t)$ for $t\in[t_i,t_{i+1})$, and the event times $\{t_i\}_{i\in\mathbb{N}}$ are to be determined by the event-triggering condition~\eqref{trigger}.

To verify condition~(iii) of Theorem~\ref{Theorem}, it yields from the dynamics of system~\eqref{example} that
\begin{align}
\dot{V}_1(t) =& 2x_1(t)\dot{x}_1(t) + 2x_2(t)\dot{x}_2(t) \cr
             =& (2-2\omega_2)x_1(t)x_2(t) +(2-2\omega_2) x^3_1(t)x_2(t) -2x^2_1(t)\cr
              & -2\omega_1 x^2_2(t) -2x_1(t)\epsilon_1(t) -2\omega_3 x_2(t)x_1(t-1)
\end{align}
and
\begin{align}
\dot{V}_2(t) =& -\zeta V_2(t) +\delta V_1(t) - \delta e^{-\zeta} (x^2_1(t-1)+ x^2_2(t-1)),
\end{align}
then,
\begin{align}
\dot{V}(t) \leq & (2-2\omega_2)x_1(t)x_2(t) +(2-2\omega_2) x^3_1(t)x_2(t) +\epsilon^2_1(t) \cr
                & -\delta e^{-\zeta} x^2_2(t-1) +(\omega_3-\delta e^{-\zeta})x^2_1(t-1) -\zeta V(t) \cr
                & +(-2+1+\zeta+\delta)x^2_1(t) +(-2\omega_1+\omega_3+\lambda+\delta)x^2_2(t) \cr
           \leq & -0.28 V(t) + \|\epsilon(t)\|^2
\end{align}
where $\epsilon(t)=(\epsilon_1(t),\epsilon_2(t))^\top=x(t_i)-x(t)$ for $t\in[t_i,t_{i+1})$. Therefore, condition~(iii) is satisfied with $\mu=0.28$ and $\chi(\|\epsilon\|)=\|\epsilon(t)\|^2$.

Based on the above analysis, the event-triggering condition~\eqref{trigger} can be written as follows:
\begin{eqnarray}\label{example.trigger}
\|\epsilon(t)\| = \sigma \|x(t)\|^2 +\left( a e^{-b(t-t_0)} \right)^2
\end{eqnarray}
where $\sigma< \mu$, and parameters $a$ and $b$ can be chosen arbitrarily. Theorem~\ref{Theorem2} concludes that system~\eqref{example1} with the triggering condition~\eqref{example.trigger} is globally asymptotically stable. Simulation results are shown in Fig.~\ref{fig1}, Fig.~\ref{fig2}, and Fig.~\ref{fig3} with initial condition $x(s)=(1,2)^\top$ for $s\in[-1,0]$, initial time $t_0=0$, and parameters $\sigma=0.16$, $a=1$, $b=0.14$. It should be noted that $b>\mu-\sigma=0.12$ in our simulations, and hence the results in~\cite{KZ-BG-EB:2022} are not applicable for these parameter selections.

\begin{figure}[!t]
\centering
\includegraphics[width=3.3in]{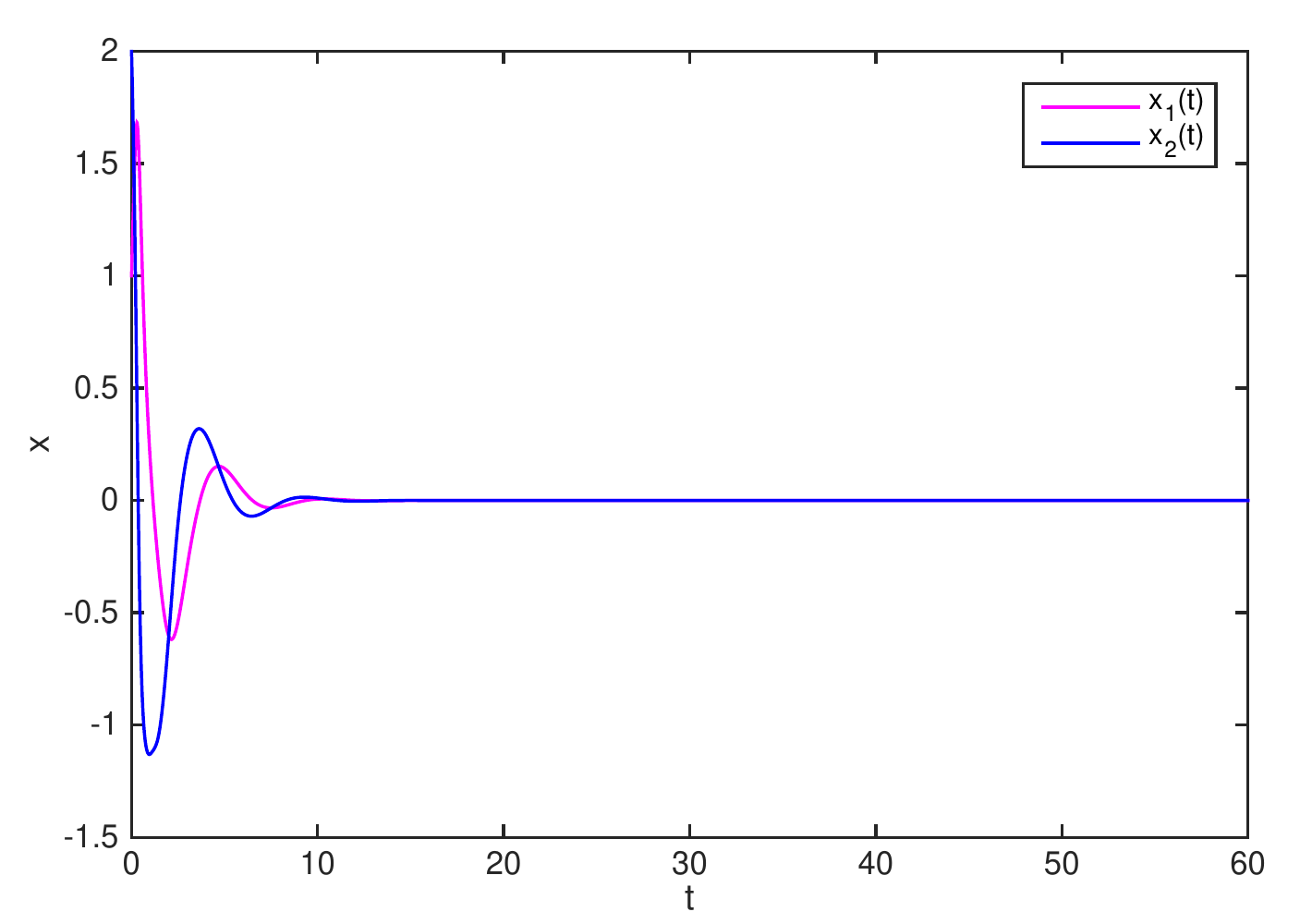}
\caption{System trajectories with the state feedback control $u(t)=-x_1(t)$. }
\label{fig1}
\end{figure}

\begin{figure}[!t]
\centering
\includegraphics[width=3.3in]{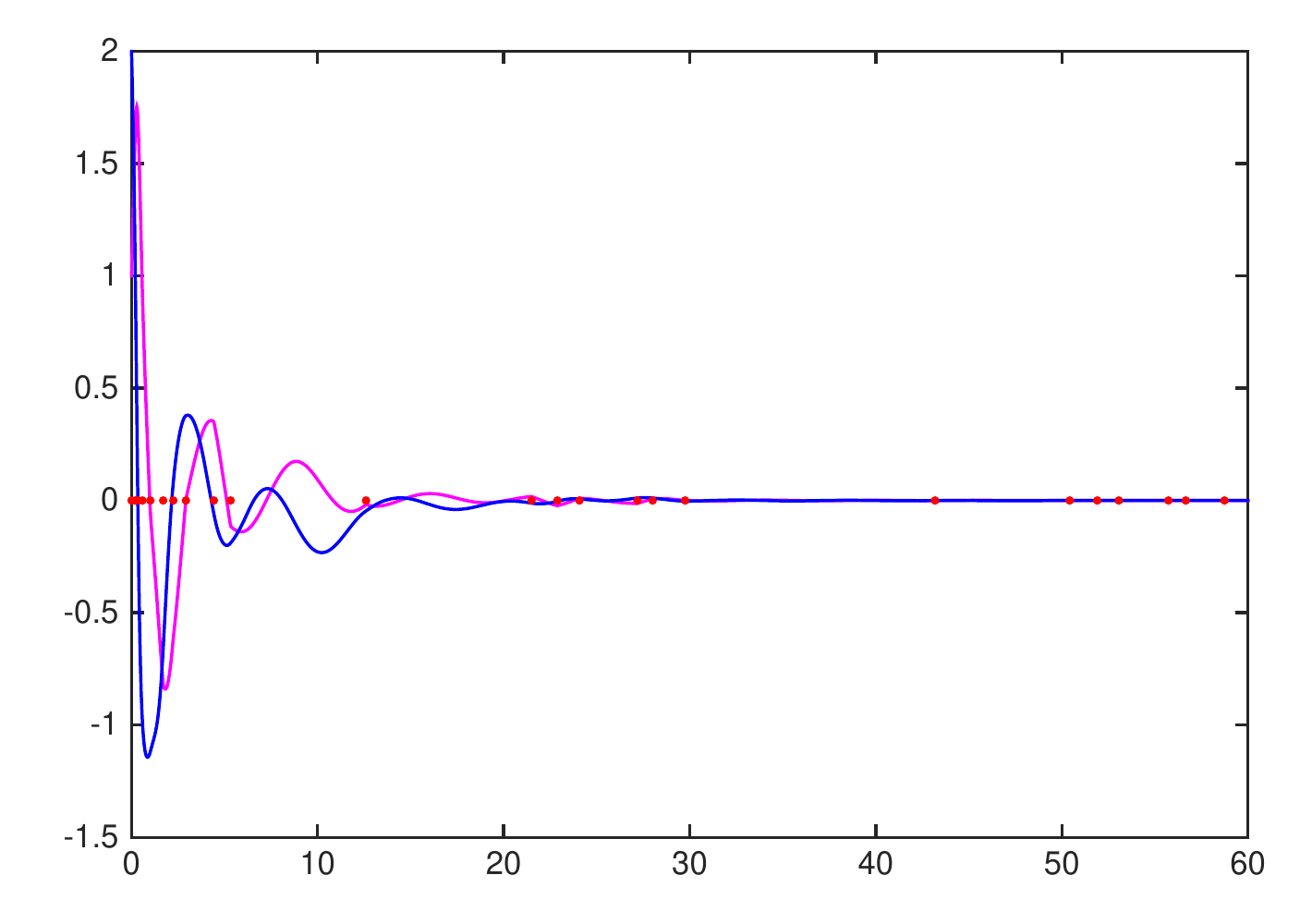}
\caption{System trajectories with the proposed event-triggered control mechanism. Red dots on the time axis indicate the event times.}
\label{fig2}
\end{figure}

\begin{figure}[!t]
\centering
\includegraphics[width=3.3in]{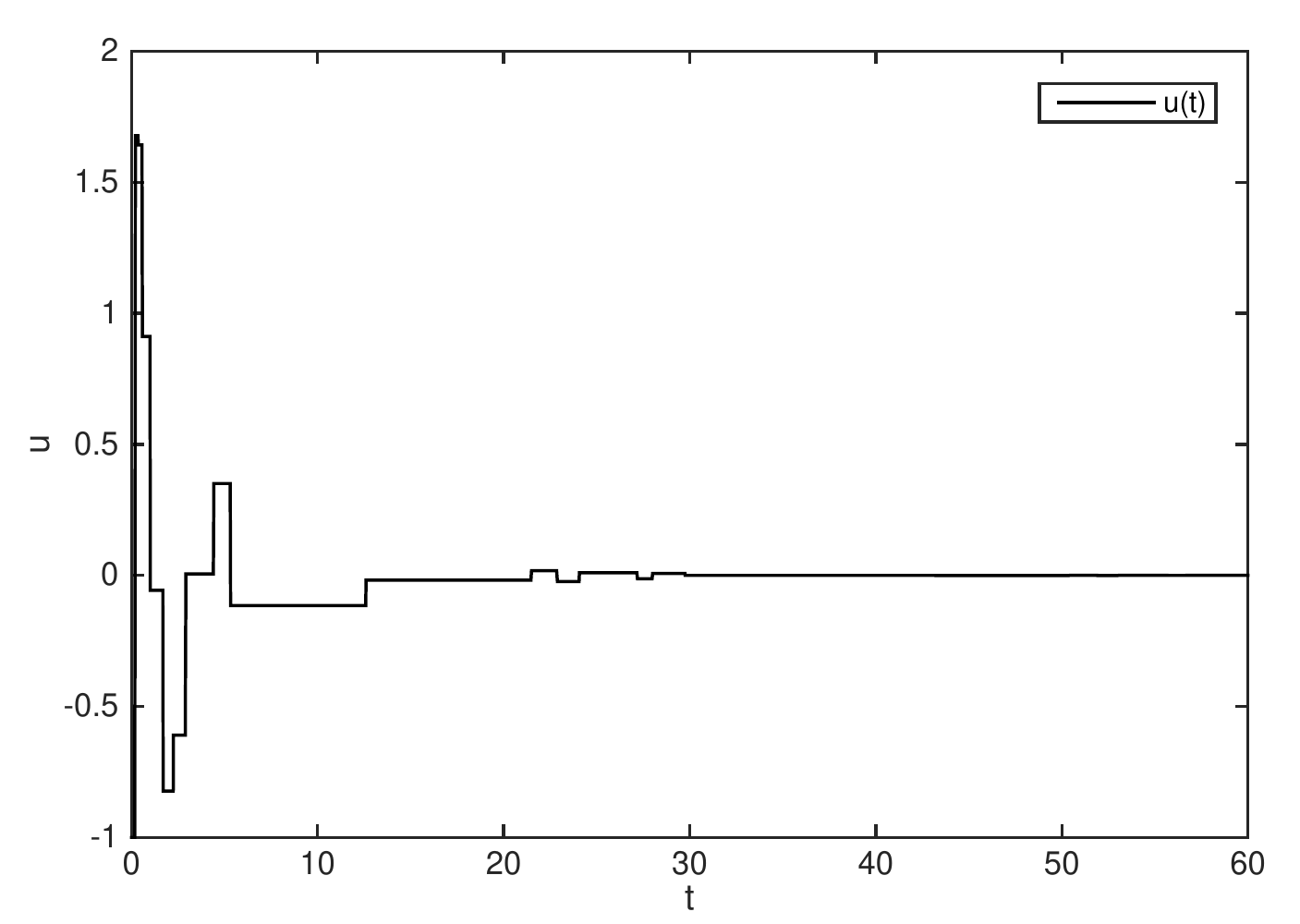}
\caption{Event-triggered control input under the triggering condition~\eqref{example.trigger}.}
\label{fig3}
\end{figure}

\section{Conclusions}\label{Sec6}
We have revisited the event-triggered control problem for time-delay systems considered in~\cite{KZ-BG-EB:2022}. It has been shown that under the sufficient conditions proposed in~\cite{KZ-BG-EB:2022} the event-triggered control system is globally asymptotically stable rather than just uniformly bounded and globally attractive. Moreover, our analysis has allowed us to arbitrarily choose a parameter in the event-triggering condition to both ensure global asymptotic stability and non-existence of Zeno behavior in the event-triggered time-delay control systems. A numerical example has been presented to verify the theoretical results.

}

\end{document}